\documentclass{amsart}
\usepackage{amsmath, amssymb,epic,graphicx,mathrsfs,enumerate}
\usepackage[all]{xy}
\usepackage{color}
\usepackage{comment}

\usepackage{amsthm}
\usepackage{amssymb}
\usepackage{latexsym}
\usepackage{longtable}
\usepackage{epsfig}
\usepackage{amsmath}
\usepackage{hhline}

\DeclareMathOperator{\inn}{Inn} 
 \DeclareMathOperator{\soc}{soc}
\DeclareMathOperator{\aut}{Aut}

 \DeclareMathOperator{\frat}{Frat}

\newtheorem{thm}{Theorem}
\newtheorem{cor}[thm]{Corollary}
 \newtheorem{lemma}[thm]{Lemma}
\newtheorem{prop}[thm]{Proposition} 
 
\newtheorem{question}[thm]{Question}

\numberwithin{equation}{section}

\renewcommand{\footnote}{\endnote}
\newcommand{\ignore}[1]{}\makeglossary

\begin{document}
	\bibliographystyle{amsplain}
	\title[The $A$-M\"{o}bius function of a finite group]{The $A$-M\"{o}bius function of a finite group}
	\author{Francesca Dalla Volta}
\address{Francesca  Dalla Volta, Dipartimento di Matematica e Applicazioni ,
	University of Milano-Bicocca, Via Cozzi 55, 20126 Milano, Italy} 
\email{francesca.dallavolta@unimib.it}
	\author{Andrea Lucchini}
 	\address{Andrea Lucchini, Universit\`a degli Studi di Padova\\  Dipartimento di Matematica \lq\lq Tullio Levi-Civita\rq\rq, Via Trieste 63, 35121 Padova, Italy}\email {lucchini@math.unipd.it}

		
\begin{abstract}
The M\"{o}bius function of the subgroup lettice of a finite group $G$ has been introduced by Hall  and applied to investigate several different questions. We propose the following generalization. Let $A$ be a  subgroup of the automorphism group $\aut(G)$ of a finite group $G$ and denote  by $\mathcal C_A(G)$ the set of $A$-conjugacy classes of subgroups of $G.$ For $H\leq G$ let
$[H]_A~=~\{~H^a ~\mid ~a\in ~A\}$ be the element of $\mathcal C_A(G)$ containing $H.$  We may define an ordering in $\mathcal C_A(G)$  in the following way:
$[H]_A\leq [K]_A$ if  $H^a\leq K$ for some $a\in A$. We 
consider  the M\"{o}bius function  $\mu_A$ of the corresponding poset and analyse its properties and possible applications.
	\end{abstract}
	\maketitle


\section{Introduction}
The M\"{o}bius function of a finite  partially ordered set (poset) $P$ is the map
$\mu_P:~ P\times P\to ~\mathbb Z$ 
satisfying $\mu_P(x,y) = 0$ unless $x \leq y,$ in which case it is defined inductively by the equations 
$\mu_P(x,x) = 1$ and $\sum_{x\leq z\leq y}\mu_P(x,z)=0$ for $x < y.$

In a celebrated paper \cite{hall}, P. Hall used for the first time the  M\"{o}bius function $\mu$ of the subgroup lattice of a finite group $G$ to investigate some properties of $G$, in particular to compute the number of generating $t$-tuples of $G$. A detailed investigation of the   properties  of the function $\mu$ associated to a finite group $G$  is given by T. Hawkes, I. M. Isaacs and M. \"{O}zaydin in \cite{isac}. In that paper, the authors also consider the M\"{o}bius function $\lambda$ of the poset  of conjugacy classes of subgroups of $G,$ where $[H] \leq [K]$ if 
$H \leq K^g$ for some $g\in G$  (see \cite[Section 7]{isac}). In particular, they propose the interesting and intriguing question of comparing the values of $\mu$ and $\lambda$.

In this paper we aim to generalize the definitions and main properties of the functions $\mu$ and $\lambda$ to a more general contest. 
Let $G$ and $A$ be a finite group and  a subgroup of the automorphism group $\aut(G)$ of $G$, respectively.  Denote by $\mathcal C_A(G)$ the set of $A$-conjugacy classes of subgroups of $G.$ For $H\leq G$ let
$[H]_A~=~\{~H^a ~\mid ~a\in ~A\}$ be the element of $\mathcal C_A(G)$ containing $H.$  We may define an ordering in $\mathcal C_A(G)$  in the following way:
$[H]_A\leq [K]_A$ if  $H^a\leq K$ for some $a\in A$; we
 consider  the M\"{o}bius function  $\mu_A$ of the corresponding poset. We will write $\mu_A(H,K)$ in place of $\mu_A([H]_A,[K]_A).$
When $A=\inn(G),$ we write $\mathcal C(G)$ and $[H]$, in place of 
$\mathcal C_{\inn(G)}(G)$ and $[H]_{\inn(G)}$.
When $A=1$, $\mu_A=\mu$ is the  M\"{o}bius function in the subgroup lattice of $G,$ introduced by P. Hall. In the case when
$A=\inn(G)$ is the group of the inner automorphism, then $\mu_{\inn(G)}=\lambda$. Note that for any subgroup $A$ of $\aut(G)$, we get $[G]_A=\{G\}$.

In Section \ref{crap}, we prove some general properties of $\mu_A$. In particular we prove the following result:

\begin{prop}\label{uguali}Let $G$ be a finite solvable group. If $G^\prime \leq K \leq G$ and $A$ is the subgroup of $\inn(G)$ obtained by considering the conjugation with the elements of $K,$ then $\lambda(H,G)=\mu_A(H,G)$ for any $H\leq G.$
\end{prop}

To illustrate the meaning of the previous proposition, consider the following example. Let $G=A_4$ be the alternating group of degree 4 and $A$ the subgroup of $\inn(G)$ induced by conjugation with the 
elements of $G^\prime \cong C_2 \times C_2.$ The posets $\mathcal C(G)$ and $\mathcal C_A(G)$ are different. For example there are three subgroups of $G$ of order 2, which are conjugated in $G$, but not $A$-conjugated. However $\lambda(H,G)=\mu_A(H,G)$ for any $H\leq G.$

In Section \ref{genhall}, we generalize some result given by Hall in \cite{hall}, about the cardinality  $\phi(G,t)$ of the set 
 $\Phi(G,t)$ 
of $t$-tuples $(g_1,\dots,g_t)$ of group elements $g_i$ such that
$G=\langle g_1,\dots,g_t\rangle$.  As observed by P. Hall, using the M\"{o}bius inversion formula, it can be proved that
\begin{equation}\label{hallfo}
\phi(G,t)=\sum_{H \leq G}\mu(H,G)|H|^t.
\end{equation}
We  generalize this formula, showing that $\phi(G,t)$ can be computed with a formula involving $\mu_A$ for any possible choice of $A.$ 

\begin{thm}\label{compfi}For any finite group $G$ and any subgroup $A$ of $\aut(G),$
	$$\phi(G,t)=\sum_{[H]_A \in \mathcal C_A(G)}\mu_A(H,G)|\cup_{a\in A}(H^a)^t|.$$
\end{thm}
If $G$ is not cyclic, then $\phi(G,1)=0$, so we obtain the following equality, involving the values of $\mu_A.$
\begin{cor}If $G$ is not cyclic, then
$$0=\sum_{[H]_A \in  \mathcal C_A(G)}\mu_A(H,G)|\cup_{a\in A}H^a|.$$
\end{cor}

Further generalizations are given in Section \ref{numbersbgs}, where we consider the function $\phi^*(G,t)$, which is an analogue of $\phi (G,t)$: actually, $\phi^*(G,t)$ denotes the cardinality of the set of of $t$-tuples $(H_1,\dots,H_t)$ of subgroups of $G$ such that
$G=\langle H_1,\dots,H_t\rangle$.
As a corollary of our formula for computing $\phi^*(G,t)$, we obtain we following unexpected result.
\begin{prop}Let $\sigma(X)$ denote the number of subgroups of a finite group $X.$ 
For any finite group $G,$ the following equality holds:
$$1=\sum_{H \leq G}\mu(H,G)\sigma(H).$$	
	\end{prop}

Finally, in  Section  \ref{lamu},  we consider one question originated from a result given by Hawkes, Isaacs and \"{O}zaydin in  \cite{isac}:  they proved
 that the equality
$$\mu(1,G) = |G^\prime| · \lambda(1,G)$$
holds for any finite solvable group $G$; later Pahlings \cite{pah} generalized the result  proving that
\begin{equation}\label{lm} \mu(H,G) = |N_{G^\prime}(H) : G^\prime \cap H| \cdot \lambda(H,G)
\end{equation}
holds for any $H \leq G$ whenever $G$ is finite and solvable.
Following \cite{dvz}, we say that $G$ satisfies the $(\mu,\lambda)$-property if
 (\ref{lm}) holds for any $H\leq G.$ Several classes of non-solvable groups satisfy the $(\mu, \lambda)$-property, for example all the minimal non-solvable groups (see \cite{dvz}). However
it is known that the $(\mu, \lambda)$-property does not hold for every finite group. For instance, it does not hold for the following 
finite almost simple groups: $A_9$, $S_9,$ $A_{10},$ $S_{10},$ $A_{11},$ $S_{11},$ $A_{12},$ $S_{12}$, $A_{13}$, $S_{13},$ $J_2$,  $PSU(3,3)$, $PSU(4,3)$, $PSU(5,2),$
$M_{12},$  $M_{23},$ $M_{24},$ $PSL(3,11),$ $HS$,  $\aut(HS),$ $He$ $\aut(H)$, $McL,$ 
$PSL(5,2),$ $G_2(4)$, $Co_3$, $P\Omega^-(8,2),$ $P\Omega^+(8,2).$ It is somehow intriguing to notice that although the $(\mu, \lambda)$-property fails for the sporadic groups  $M_{12},$ $J_2,$ $McL,$ it holds for their automorphism groups. 

We prove the following  generalization of Pahlings's result.

\begin{thm}\label{lamuso}Let  $N$ be a solvable normal subgroup of a finite group $G$. If $G/N$ satisfies the $(\mu,\lambda)$-property,
	then $G$ also satisfies the $(\mu,\lambda)$-property.
\end{thm}

An almost immediate consequence of the  previous theorem is the following.

\begin{cor}\label{psu}$PSU(3,3)$ is the smallest group which does not satisfy the $(\mu,\lambda)$ property.
\end{cor}

In the last part of Section \ref{lamu}, 
we use Theorem \ref{compfi} to deduce some consequences of the $(\mu,\lambda)$-property. In particular we prove the following theorem.

\begin{thm}Suppose that a finite group $G$ satisfies the  $(\mu,\lambda)$-property.  Then, for every positive integer $t,$ the following equality is satisfied:
$$		\sum_{[H] \in \mathcal C(G)}\lambda(H,G)\left(\frac{|H|^{t-1}|G||G^\prime  H|}{|G^\prime N_G(H)|}-|\cup_{a\in A}(H^a)^t|\right)=0.
$$
\end{thm}

Some open questions are proposed along the paper.

\section{Applying some general properties of the M\"{o}bius function}\label{crap}

 Given a poset $P$,  a closure on $P$ is a function $\bar{}:P\rightarrow P$ satisfying the following three conditions:
\begin{itemize}
	\item[a)] $x\leq\bar{x}$ for all $x\in P$;
	\item[b)] if $x,y\in P$ with $x\leq y$, then $\bar{x}\leq\bar{y}$;
	\item[c)] $\bar{\bar{x}}=\bar{x}$ for all $x\in P$.
\end{itemize}
If $\,\bar{}\,$ is a closure map on $P,$ then $\overline{P}=\left\{ x\in P|\,\bar{x}=x\right\}$
is a poset with order induced by the order on $P$. We have:
\begin{thm}[The closure theorem of Crapo \cite{crapo2}]\label{crap1} Let $P$ be a finite
	poset and let \ $\bar{}: P \to P$  be a closure map. Fix $x,y \in P$
	such that $y \in \overline P.$ Then
	$$\sum_{\bar z=y}\mu_P(x,z)=\begin{cases}\mu_{\bar P}(x,y) {\text{ if }}x = \bar x\\0
	{\text { otherwise.}}
	\end{cases}$$
\end{thm}

In \cite{hall}, P. Hall proved that if $H < G$, then $\mu(H,G)\neq 0$ only if $H$ is an intersection of maximal subgroups of $G$.
Using the previous theorem, the following more general statement can be obtained.
\begin{prop}\label{intmax}
	If $H<G$ and $\mu_A(H,G)\neq 0,$ then $H$ can be obtained as intersection of maximal subgroups of $G.$
\end{prop}
\begin{proof}
Let $H$ be a proper subgroup of $G$ and let $\overline H$ be the intersection of the maximal subgroups of $G$ containing $H$. The map $[H]_A \mapsto[\overline H]_A$ is a well defined closure map on $\mathcal C_A(G),$ so the conclusion follows immediately from Theorem \ref{crap1}.
\end{proof}

An element $a$ of a poset $\mathcal P$ is called conjunctive
if the pair $\{a, x\}$ has a least upper bound, written $a\vee x,$ for each $x \in \mathcal  P.$

\begin{lemma}\cite[Lemma 2.7]{isac}\label{congi}
	Let $\mathcal P$ be a poset with a least element $0,$ and let $a > 0$
	be a conjunctive element of $\mathcal P.$ Then, for each $b> a,$ we have
	$$\sum_{a\vee x=b}\mu_{\mathcal P}(x)=0.$$
\end{lemma}
From the above Lemma  \ref{congi}, it follows easily the following Lemma \ref{cocra}  which, together with Lemma \ref{normalgen} and Lemma \ref{quiquo} allow us to prove Proposition \ref{uguali} 
\begin{lemma}\label{cocra}Let $N$ be an $A$-invariant normal subgroup of $G$ and $H\leq G.$
	If $H < HN < G,$
	then $$\mu_A(H,G)=-\sum_{[Y]_A \in \mathcal S_A(H,N)}\mu_A(H,Y),$$ with
	$\mathcal S_A(H,N)=\{[Y]_A \in \mathcal C_A(G) \mid [H]_A\leq [Y]_A  < [G]_A\text { and } YN=G\}$.
	\end{lemma}
\begin{proof}
	Let $\mathcal P$ be the interval $\{[K]_A \in \mathcal C_G(A) \mid [H]_A \leq [K]_A \leq [G]_A]\}.$
	Notice that $[HN]_A$ is a conjunctive element of $\mathcal P.$ Indeed $[HN]_A \vee [K]_A = [KN]_A$ for every $[K]_A \in \mathcal P.$ So the conclusion follows immediately from Lemma \ref{congi}
	\end{proof}

 \begin{lemma}\label{normalgen}
	Let $K$ and $A$ be a subgroup of $G$ and the subgroup of $\inn(G)$ induced by the conjugation with the elements of $K$, respectively. Assume that $N$ is an abelian minimal normal subgroup of $G$ contained in $K$ and
$H < HN \leq  G.$ 
Then $$\mu_A(H,G)=-\mu_A(HN,G)
\gamma_A(N,H),$$ where $\gamma_A(N,H)$ is the number of $A$-conjugacy classes  of complements of $N$ in $G$ containing $H.$
\end{lemma}
\begin{proof}
	If $HN=G,$ then $H$ is a maximal subgroup of $G,$ hence $\mu_A(H,G)=-1,$
	while $\mu_A(HN,G)=\mu_A(G,G)=1$ and 
	$\gamma_A(N,H)=1$, so the statement is true. So we may assume $HN < G$ and
 apply Lemma \ref{cocra}. Suppose $[Y]_A \in \mathcal S_A(H,N).$
	Let $$\begin{aligned}\mathcal C&=\{J\leq G \mid H\leq J \leq Y\}, \quad \mathcal D=\{L\leq G \mid HN\leq L\}
\\\mathcal C_A&=\{[J]_A\in \mathcal C_A(G) \mid [H]_A\leq [J]_A\leq [Y]_A\}, \ \mathcal D_A=\{[L]_A\in \mathcal C_A(G) \mid [HN]_A\leq [L]_A\}.
\end{aligned}$$ The map $\eta: \mathcal C \to \mathcal D$ sending $J$ to $JN$ is an order preserving bijection. Clearly, if $J_2=J_1^x$ for some $x \in K,$ then $\gamma(J_2)=NJ_2=NJ_1^x=(NJ_1)^x=(\gamma(J_1))^x.$ Conversely assume
$\gamma(J_2)=(\gamma(J_1))^x$ with $x \in K.$ Since $YN=G,$ $x=yn$ with $n \in N$ and $y \in Y\cap K.$ Thus $J_2N=(J_1N)^x=(J_1N)^y$
and consequently $J_2=J_2N \cap Y= (J_1N)^y \cap Y= (J_1N\cap Y)^y=J_1^y.$ It follows that $\eta$ induces an order preserving bijection from $\mathcal C_A$ to $\mathcal D_A,$ but then $\mu_A(H,Y)=\mu_A(HN,YN)=\mu_A(HN,G).$
\end{proof}

The statement  of the previous lemma leads to the following open question.

\begin{question}
	Let $G$ be a finite group, $A\leq \aut(G)$ and $N$ an $A$-invariant normal subgroup of $G$. Does $\mu_A(HN,G)$ divide $\mu_A(H,G)$ for every $H\leq G?$
\end{question}

The following lemma is straightforward.
\begin{lemma}\label{quiquo}
	Let $A$ be a subgroup of $\aut(G)$ and
	$N$  an $A$-invariant normal subgroup of $G$. Every $a\in A$ induces an automorphism $\overline a$ of $G/N$. Let $\overline A=\{\overline a \mid a \in A\}.$ 
	Then, for any $H\leq G,$ $\mu_A(HN,G)=\mu_{\overline A}(HN/N,G/N).$
\end{lemma}

\begin{proof}[Proof of Proposition \ref{uguali}] We work by induction on $|G|\cdot |G:H|$.   The statement is true if $G$ is abelian. Assume that $G^\prime$ contains a minimal normal subgroup, say $N,$ of $G.$
	If $N\leq H,$  then, by Lemma \ref{quiquo} $$\lambda(H,G)=\lambda(H/N,G/N)=\mu_{\overline A}(H/N,G/N)=\mu_A(H,G).$$
	So we may assume $N\not\leq H.$ If $H$ is not an intersection of maximal subgroups of $G,$ then $\lambda(H,G)=\mu_A(H,G)=0$. Suppose  $H=M_1\cap \dots \cap M_t$ where $M_1,\dots, M_t$ are maximal subgroups of $G.$ In particular $N$ is not contained in $M_i$ for some $i,$ so
	$M_i$ is a complement of $N$ in $G$ containing $H$ and $N\cap H=1.$ 
	By Lemma \ref{normalgen},	we have $$\lambda(H,G)=-\lambda(HN,G)\gamma(N,H), \quad \mu_A(H,G)=-\mu_A(HN,G)\gamma_A(N,H), $$
	where $\gamma(N,H)$ is the number of conjugacy classes of complements of $N$ in $G$ containing $H$ and $\gamma_A(N,H)$ is the number of $A$-conjugacy classes of these complements.
	Suppose that $K_1, K_2$ are two conjugated complements if $N$ in $G$ containing $H.$ 
Then $K_2=K_1^n$
for some  $n\in N_N(H).$ Since $N \leq G'\leq K,$ it follows $\gamma(N,H)=\gamma_A(N,H).$ Moreover, by induction, $\lambda(HN,G)=\mu_A(HN,G),$ hence we conclude $\lambda(H,G)=
	\mu_A(H,G).$
\end{proof}

\section{Generalizing a formula of Philip Hall}\label{genhall}
We begin with introducing the functions $\Psi _A(H,t)$ and $\psi _A(H,t)$, analogue of $\Phi (H,t)$ and $\phi (H,t)$ in the general case of any possible subgroup $A$ of $\aut (G)$.

For any $H \in \mathcal C_A(G)$ and any positive integer $t,$ let
\begin{enumerate}
	\item $\Omega_A(H,t)=\bigcup_{a\in A}(H^a)^t;$
	\item $\omega_A(H,t)=|\Omega_A(H,t)|;$
	\item $\Psi_A(H,t)=\{(g_1,\dots,g_t)\in G^t \mid \langle g_1,\dots,g_t\rangle=H^a \text { for some } a \in A\};$
	\item $\psi_A(H,t)=|\Psi_A(H,t)|.$
\end{enumerate}
If $(x_1,\dots,x_t)\in \Omega_A(H,t),$ then
$\langle x_1,\dots,x_t\rangle \leq H^a$ for some $a\in A,$ hence
$\langle x_1,\dots,x_t\rangle=K$ for some $K\leq G$ with $[K]_A\leq [H]_A.$ Thus 
$$\sum_{[K]\leq_A [H]}\psi_A(K,t)=\omega_A(H,t)$$ and therefore, by the
M\"{o}bius inversion formula,
$$\sum_{[H]\in \mathcal C_A(G)}\mu_A(H,G)\omega_A(H,t)=\psi_A(G,t).$$
On the other hand $\psi_A(G,t)=\phi(G,t)$ so we have proved the following formula.
\begin{thm}\label{generalehall}For any finite group $G$ and any subgroup $A$ of $\aut(G),$
$$\phi(G,t)=\sum_{[H] \in \mathcal C_A(G)}\mu_A(H,G)\omega_A(H,t).$$
\end{thm}
Notice that if $A=1,$ then $\omega_A(H,t)=|H^t|$, so that the result by Hall given in (\ref{hallfo}) is a particular case of the previous theorem.

\begin{cor}If $G$ is not cyclic, then
	$$0=\phi(G,1)=\sum_{[H] \in  \mathcal C_A(G)}\mu_A(H,G)\omega_A(H,1).$$
\end{cor}

Taking $A=\inn(G),$ we deduce in particular
 that if
$G$ is not cyclic, then 
$$\sum_{H \in \mathcal C(H)}\lambda(H,G)\omega_{\inn(G)}(H,1)=\sum_{H \in \mathcal C(H)}\lambda(H,G)|\cup_g H^g|=0.$$
\

For example, if $G=S_4,$ then the values of $\lambda(H,G)$ and $|\cup_g H^g|$ are as in the following table and
$24-12-16-15+4+9+7-1=0.$

\begin{center}
	\begin{tabular}{|c|c|c|}
		\hline
		&$\lambda(H,G)$&$|\cup_g H^g|$\\
		\hline
		$S_4$ &          1 &         24 \\
		\hline
		$A_4$ &         -1 &         12 \\
		\hline
		$D_4$ &         -1 &         16 \\
		\hline
		$S_3$ &         -1 &         15 \\
		\hline	
		$K$ &          1 &          4 \\
		\hline
		$\langle (1,2,3,4)\rangle$ &          0 &         10 \\
		\hline
		$\langle (1,2,3)\rangle$ &          1 &          9 \\
		\hline
		$\langle (1,2\rangle$ &          1 &          7 \\
		\hline
		$\langle (1,2)(3,4)\rangle$ &          0 &          4 \\
		\hline
		1 &         -1 &          1 \\
		\hline
	\end{tabular}  
\end{center}

\

If $G=A_5,$ then the values of $\lambda(H,G)$, $\omega_{\inn(G)}(H,1)=|\cup_g H^g|$, $\omega_{\inn(G)}(H,2)=|\cup_g (H^g)^2|$ (taking only the subgroups $H$ with $\lambda(H,G)\neq 0$) are as in the following table and 60-36-36-40+21+32-1=0.

\

\begin{center}
	\begin{tabular}{|c|c|c|c|}
		\hline
		&$\lambda(H,G)$&$|\cup_g H^g|$&$|\cup_g (H^g)^2|$\\
		\hline
		$A_5$ &          1 &         60&3600 \\
		\hline
		$A_4$ &         -1 &        36& 636 \\
		\hline
		$S_3$ &         -1 &         36& 306 \\
		\hline
		$D_5$ &         -1 &         40&550 \\
		\hline	
		$\langle (1,2,3)\rangle$ &          1 &          21&81 \\
		\hline
		$\langle (1,2)(3,4)\rangle$ &          2 &          16&46 \\
		\hline
		1 &         -1 &          1&1 \\
		\hline
	\end{tabular}  
\end{center}
Moreover
$$
3600-636-306-550+81+2\cdot 46-1=2280=\frac{19}{30}\cdot 3600=\phi(A_5,2).
$$

\

If $G=D_{p}=\langle a, b \mid a^p=1, b^2=1, a^b=a^{-1}\rangle$ and $p$ is an odd prime, then  the behaviour of the subgroups in $\mathcal C(G)$ is described by the following table.

\

\begin{center}
	\begin{tabular}{|c|c|c|}
		\hline
		&$\lambda(H,G)$&$|\cup_g H^g|$\\
		\hline
		$D_p$ &          1 &         $2p$ \\
		\hline
		$\langle a\rangle $ &         -1 &       $p$ \\
		\hline
		$\langle b\rangle$ &         -1 &         $p+1$ \\
		\hline
		1 &         -1 &          1 \\
		\hline
	\end{tabular}  
\end{center}

\

Another interesting example is given by considering $G=C_p^n$ and $A=\aut(G).$ Let $H\cong C_p^{n-1}$ be a maximal subgroup of $G.$ Then, for $K\leq G$,  $\mu_A(K,G)\neq 0$ if and only if either $[K]_A=[G]_A$ or $[K]_A=[H]_A.$ Clearly $\cup_{\alpha\in \aut(G)} H^\alpha=G$ so 
$\mu_A(G,G)\omega_A(G,1)-\mu_A(H,G)\omega_A(H,1)=|G|-|G|=0.$ More in general $\Omega_A(H,t)$
is the set of $t$-tuples $(x_1,\dots,x_t)$ such that
$(x_1,\dots,x_t) \in K^t$ for some maximal subgroup $K$ of $G$, so
${\mu_A(G,G)}\omega_A(G,t)-{\mu_A(H,G)}\omega_A(H,t)=|G|^t-\omega_A(H,t)$ is the number of generating $t$-tuples of $G.$

\

Another generalization of (\ref{hallfo}), essentially due to Gasch{\"u}tz,  has been described by Brown in \cite[Section 2.2]{brown}.
Let $N$ be a normal subgroup of $G$ and suppose that $G/N$ admits $t$ generators for some  integer $t$. Let $y=(y_1,\dots,y_t)$ be a generating $t$-tuple of $G/N$ and denote by $P(G,N,t)$ the probability that a random lift of $y$ to a $t$-tuple of $G$  generates $G$. Then $P(G,N,t)=\phi(G,N,t)|N|^t$, where $\phi(G,N,t)$ is the number of generating $t$-tuples of $G$ lying over $y$. Using again the M\"{o}bius inversion formula it can be proved:
\begin{equation}\label{relhall}
\phi(G,N,t)=\sum_{H \leq G, HN=G}\mu(H,G)|H\cap N|^t.
\end{equation}
This formula can be generalized in our contest in the following way:
\begin{thm}
Let $N$ be an $A$-invariant normal subgroup of $G$ and
fix  $g_1,\dots,g_t \in G$ with the property that $G=\langle g_1,\dots,g_t\rangle N.$  Define
\begin{itemize}
	\item $\Omega_A(H,N,t)=\{(n_1,\dots,n_t)\mid \langle g_1n_1,\dots,g_tn_t\rangle  \leq H^a \text { for some } a\in A\};$
\item $\omega_A(H,N)=|\Omega_A(H,N,t)|$
\end{itemize}
and let $\mathcal C_A(G,N)=\{[H]_A \in \mathcal C_A(G) \mid HN=G\}.$
Then
$$\phi(G,N,t)=\sum_{[H]_A \in \mathcal C_A(G,N)} \mu_A(H,G)\omega_A(H,N,t).$$
\end{thm}
\begin{proof}
Fix  $g_1,\dots,g_t \in G$ with the property that $G=\langle g_1,\dots,g_t\rangle N.$ Then $\phi(G,N,t)$ is the cardinality of the set
$$\Phi(G,N,g_1,\dots,g_t)=\{(n_1,\dots,n_t)\in N^t \mid
\langle g_1n_1, \dots, g_tn_t\rangle=G\}.$$ Set:
$$\begin{aligned}&\Psi_A(H,N,g_1,\dots,g_t)=\{(n_1,\dots,n_t)\!\in\! N^t \mid \langle g_1n_1,\dots,g_tn_t\rangle\!=\! H^a \text { for some } a \in A\};\\
&\psi_A(H,N,t)=|\Psi_A(H,N,g_1,\dots,g_t)|.
\end{aligned}
$$
Notice that $\omega_A(H,N,t)\neq 0$ if and only if $[H]_A\in  \mathcal C_A(G,N).$ If $(n_1,\dots,n_t)\in \Omega_A(H,N,t),$ then
$\langle g_1n_1,\dots,g_tn_t\rangle \leq H^a$ for some $a\in A,$ and
$\langle g_1n_1,\dots,g_tn_t\rangle$ $=K$ for some $K\leq G$ with $[K]_A\leq [H]_A.$ Thus 
$$\sum_{[K]_A\leq [H]_A}\psi_A(K,N,t)=\omega_A(H,N,t)$$ and therefore, by the
M\"{o}bius inversion formula
$$\sum_{[H]\in \mathcal C_A(G,N)}\mu_A(H,G)\omega_A(H,N,t)=\psi_A(G,N,t)=\phi(G,N,t)\qedhere$$
\end{proof}

\section{Another application of M\"{o}bius inversion formula}\label{numbersbgs}

Denote by $\Phi^*(G,t)$ the set
 of $t$-tuples $(H_1,\dots,H_t)$ of subgroups of $G$ such that
$G=\langle H_1,\dots,H_t\rangle$ and by $\phi^*(G,t)$ the cardinality of this set.
For any $H \in \mathcal C_A(G)$ and any positive integer $t,$ let
\begin{enumerate}
	\item $\Sigma_A(H,t)=\{(H_1,\dots,H_t)\mid \langle H_1,\dots,H_t \rangle \leq H^a \text{ for some } a \in A\};$
	\item $\sigma_A(H,t)=|\Sigma_A(H,t)|;$
	\item $\Gamma_A(H,t)=\{(H_1,\dots,H_t) \mid \langle H_1,\dots,H_t\rangle=H^a \text { for some } a \in A\};$
	\item $\gamma_A(H,t)=|\Gamma_A(H,t)|.$
\end{enumerate}
\begin{thm}
	$$\phi^*(G,t)=\sum_{[H]\in \mathcal C_A(G)}\mu_A(H,G)\sigma_A(H,t).$$
\end{thm}
\begin{proof}
If $(H_1,\dots,H_t)\in \Sigma_A(H,t),$ then
$\langle H_1,\dots,H_t\rangle=K$ for some $K\leq G$ with $[K]_A\leq [H]_A.$ Thus 
$$\sum_{[K]\leq_A [H]}\gamma_A(K,t)=\sigma_A(H,t)$$ and therefore, by the
M\"{o}bius inversion formula,
$$\sum_{[H]\in \mathcal C_A(G)}\mu_A(H,G)\sigma_A(H,t)=\gamma_A(G,t)=\phi^*(G,t).\qedhere$$
\end{proof}

In the particular case when $A=1,$ $\sigma_A(H,t)=\sigma(H)^t,$ denoting with $\sigma(H)$ the number of subgroups of $H.$ So we obtain the following corollary:

\begin{cor}
$$\phi^*(G,t)=\sum_{H \leq G}\mu(H,G)\sigma(H)^t.$$
\end{cor}

Clearly $\Sigma^*(G,t)=\{G\},$ so $\phi^*(G,1)=1$ and therefore it follows:

\begin{cor}
	$$1=\sum_{H \in \mathcal \mathcal H_A}\mu_A(H,G)\sigma_A(H,1).$$
\end{cor}
In particular:
\begin{cor}\label{somma1}
	$$1=\sum_{H \leq G}\mu(H,G)\sigma(H).$$
\end{cor}

For example, if $G=A_5$ then the subgroups of $G$ with $\mu(H,G)\neq 0$ are listed in the following table (where $\kappa(H,G)$ denote the numbers of conjugate of $H$ in $G$).

\begin{center}
	\begin{tabular}{|c|c|c|c|}
		\hline
		&$\mu(H,G)$&$\kappa(H,G)$&$\sigma(H)$\\
		\hline
		$A_5$ &          1 &         1&59 \\
		\hline
		$A_4$ &         -1 &        5& 10 \\
		\hline
		$S_3$ &         -1 &        10& 6 \\
		\hline
		$D_5$ &         -1 &         6&8 \\
		\hline	
		$\langle (1,2,3)\rangle$ &          2 &         10&2 \\
		\hline
		$\langle (1,2)(3,4)\rangle$ &          4 &          15&2 \\
		\hline
	1 &         -60 &          1&1 \\
		\hline
	\end{tabular}  
\end{center}
According with Corollary \ref{somma1},
$1=59-5\cdot 10-10\cdot 6-6\cdot 8 +2\cdot 10\cdot 2 +4\cdot 15 \cdot 2-60.$

\

For a finite group $G$, denote by $P(G,t)$ and $P^*(G,t)$ the probability of generating $G$ with, respectively, $t$ elements or $t$ subgroups. It can be easily seen that $P(G,t)=P(G/\frat(G),t),$ but in general $P^*(G,t)\neq P^*(G/\frat(G),t).$
For example, if $G\cong C_{p^a}$,  then $G$ and $H\cong C_{p^{a-1}}$ are the unique subgroups of $G$ with non trivial M\"{o}bius number and therefore
$$\begin{aligned}P(G,t)&=\frac{|G|^t-|H|^t}{|G|^t}=1-\frac{1}{p^t},\\
	P^*(G,t)&=\frac{\sigma(G)^t-\sigma(H)^t}{\sigma(G)^t}=1-\frac{a^t}{(a+1)^t}.
	\end{aligned}$$
So $P(G,t)$ is independent on $a$, while
$P^*(G,t)$ tends to 0 when $a$ tends to infinity.

\section{The $(\mu,\lambda)$-property}\label{lamu}

 \begin{proof}[Proof of Thereom \ref{lamuso}] Working by induction on the order of $G,$ it suffices to prove the statement in the particular case when $N$ is an abelian minimal normal subgroup of $G.$
 	Let $H$ be  a subgroup of $G.$ If $N\leq H,$ then $$\begin{aligned}\mu(H,G)&=\mu(H/N,G/N)=\lambda(H/N,G/N)|N_{G^\prime N/N}(H/N):H/N \cap G^\prime N/N|\\&=\lambda(H,G)|N_{G^\prime N}(H):H \cap G^\prime N|=\lambda(H,G)|NN_{G^\prime }(H):N(H \cap G^\prime)|\\&
 	=\lambda(H,G)\frac{|N_{G^\prime }(H):H \cap G^\prime|}{|N\cap N_{G^\prime}(H):N\cap H\cap G^\prime|}
 	=\lambda(H,G)\frac{|N_{G^\prime }(H):H \cap G^\prime|}{|N\cap G^\prime:N\cap G^\prime|}\\&=\lambda(H,G){|N_{G^\prime }(H):H \cap G^\prime|}.
 	\end{aligned}$$
 	So we may assume $N\not\leq H.$ If $H$ is not an intersection of maximal subgroups of $G,$ then $\mu(G,H)=\lambda(G,H)=0$. So we may assume $H=M_1\cap \dots \cap M_t$ where $M_1,\dots, M_t$ are maximal subgroups of $G.$ Since $N$ is not contained in $H,$ then $N$ is not contained in $M_i$ for some $i,$ but then
 	$M_i$ is a complement of $N$ in $G$ containing $H$ and $N\cap H=1.$ If $g \in N_G(HN),$ then $g=xn$ with $x \in M_i$ and $n\in N.$ In particular $H^x\leq HN \cap M_i=H(N\cap M_i)=H$, so $N_G(HN)=N_G(H)N.$
 By Lemma \ref{normalgen},	we have $$\frac{\mu(H,G)}{\lambda(H,G)}=\frac{\mu(HN,G)}{\lambda(HN,G)}\frac{\kappa}{\delta}=
 	|N_{G^\prime N}(HN):HN\cap G^\prime N|\frac{\kappa}{\delta}=|NN_{G^\prime}(H):HN\cap G^\prime N|\frac{\kappa}{\delta}	$$
 	where $k$ is the number of complements of $N$ in $G$ containing $H$ and $\delta$ is the number of conjugacy classes of these complements.
 	First assume that $N\leq Z(G).$ Then $\kappa=\delta$, $G^\prime=M_i^\prime\leq M_i,$ $N\cap G^\prime=1$ and
 	$$\begin{aligned}\frac{\mu(H,G)}{\lambda(H,G)}&=|NN_{G^\prime}(H):HN\cap G^\prime N|\frac{\kappa}{\delta}=|NN_{G^\prime}(H):HN\cap G^\prime N|\\&=|NN_{G^\prime}(H):N(H\cap G^\prime)|=
 	|N_{G^\prime}(H):H\cap G^\prime|.\end{aligned}$$
 	Finally assume $N\not\leq Z(G).$ Then $N \leq G^\prime$, $\kappa/\delta=|N_N(H)|$ and
 	$$\begin{aligned}\frac{\mu(H,G)}{\lambda(H,G)}&=|NN_{G^\prime}(H):HN\cap G^\prime N|\frac{\kappa}{\delta}=|NN_{G^\prime}(H):N(H\cap G^\prime)||N_N(H)|\\&=\frac{|N||N_{G^\prime}(H)|}{|N_N(H)|}\frac{|N_N(H)|}{|N||H\cap G^\prime|}=
 		|N_{G^\prime}(H):H\cap G^\prime|.\qedhere\end{aligned}$$
 \end{proof}

\begin{proof}[Proof of Corollary \ref{psu}]
	Suppose that $G$ has minimal order with respect to the property that $G$ does not 
	satisfy the $(\mu,\lambda)$ property. By the previous proposition, $G$ contains no abelian minimal normal subgroup and therefore $\soc(G)=S_1\times\dots \times S_t$ is a direct product of nonabelian finite simple groups. If $|G|\leq |PSU(3,3)|=6048$, then either $t=1$ or
	$G=\soc(G)=A_5\times A_5.$  So it suffices to check that $A_5\times A_5$ and any almost simple group of order at most 6048 satisfies the $(\mu,\lambda)$ property. Since, for
	every $H\leq G,$ $\lambda(H,G)$ and
	$\mu(H,G)$ can be computed
	from the table of marks of $G$ (see
	\cite[Proposition 1]{pah}), this task can be easely completed using the library of table of marks available in GAP \cite{GAP4}.
\end{proof}
We may use Theorem \ref{generalehall} to deduce some consequences of the $(\mu,\lambda)$-property.

\begin{thm}Suppose that a finite group $G$ satisfies the  $(\mu,\lambda)$-property.  Then
\begin{equation}\label{strana}
\sum_{[H] \in \mathcal C(G)}\lambda(H,G)\left(\frac{|H|^{t-1}|G||G^\prime  H|}{|G^\prime N_G(H)|}-\omega(H,t)\right)=0.
\end{equation}
\end{thm}
\begin{proof}
By Theorem \ref{generalehall}, $$\begin{aligned}\sum_{H \in \mathcal C(G)}\lambda(H,G)\omega(H,t)&=\phi(G,t)=\sum_{H\leq G}\mu(H,G)|H|^t\\&=
\sum_{H \in \mathcal C(G)}\mu(H,G)|G:N_G(H)||H^t|\\
&=
\sum_{H \in \mathcal C(G)}\lambda(H,G)|N_{G^\prime}(H) : G^\prime \cap H||G:N_G(H)||H^t|\\
&=\sum_{H \in \mathcal C(G)}\lambda(H,G)\frac{|H|^{t}|G||N_{G^\prime}(H)|}{|G^\prime \cap H||N_G(H)|}\\&=
\sum_{H \in \mathcal C(G)}\lambda(H,G)\frac{|H|^{t-1}|G||G^\prime  H|}{|G^\prime N_G(H)|}. \qedhere
\end{aligned}$$
\end{proof}
A natural question is whether (\ref{strana}) is also a sufficient condition for  the $(\mu,\lambda)$-property. For any $H\leq G,$ set
$\mu^*(H,G)=|N_{G^\prime}(H) : G^\prime \cap H|\lambda(H,G)$. 
The validity of $(\ref{strana})$ is equivalent to
$$\sum_{H \in \mathcal C(G)}\lambda(H,G)\omega(H,t)-\sum_{H \in \mathcal C(G)}\mu^*(H,G)|H|^t|G:N_G(H)|=0.$$
In any case we must have
$$\sum_{H \in \mathcal C(G)}\lambda(H,G)\omega(H,t)-\sum_{H \in \mathcal C(G)}\mu(H,G)|H|^t|G:N_G(H)|=0.$$
So  $(\ref{strana})$ is equivalent to
$$\sum_{H \in \mathcal C(G)}\frac{(\mu(H,G)-\mu^*(H,G))|H|^t}{|N_G(H)| }=0.$$
Let $\mathcal T=\{ [H] \in \mathcal C(G) \mid \mu(H,G)\neq \mu^*(H,G)\}.$
Then $(\ref{strana})$ is true if and only if
\begin{equation}\label{presunti}\sum_{[H] \in \mathcal T}\frac{(\mu(H,G)-\mu^*(H,G))|H|^t}{|N_G(H)| }=0.
\end{equation}

\

For example, if $G=PSU(3,3),$ then $\mathcal T$ consists of four conjugacy classes of subgroups and the corresponding values are given by the following table:

\

\begin{center}
	\begin{tabular}{|c|c|c|c|c|}
		\hline
		$\mu(H,G)$&$\mu^*(H,G)$&$|H|$&$|N_G(H)|$\\
		\hline
		-48 &  0 &   2 &  96 \\
		\hline
		3 &  0 &   6 &  18\\
		\hline
		0 &  -4 &   8 &  32 \\
		\hline
		1 &  2 &   24 &  24 \\
		\hline
	\end{tabular}  
\end{center}
In this case (\ref{presunti}) is equivalent to
$$2^{t-1}-6^{t-1}-8^{t-1}+24^{t-1}=0$$
which is true only if $t=1.$

\

For any positive integer $n$ let
$$\tau(n)=\sum_{H \in \mathcal T, |H|=n}\frac{\mu(H,g)-\mu^*(H,G)}{|N_G(H)| }.$$

\begin{prop}
	A finite  group $G$ satisfies
	(\ref{strana})
	for any positive integer $t$ if and only if $\tau(n)=0$ for any $\in \mathbb N.$
\end{prop}

\begin{question} Does $\tau(n)=0$ for any $n\in \mathbb N$ implies $\mu^*(H,G)=\mu(H,G)$ for any $H\leq G?$
\end{question}

For any $H\leq G,$ consider 
$$\alpha(H,t)=\frac{|H|^{t-1}|G||G^\prime  H|}{|G^\prime N_G(H)|},\quad \beta(H,t)=\alpha(H,t)-\omega(H,t).$$
Let $\mathcal C^*(G)=\{[H] \in \mathcal C(H) \mid [H] < [G] \text { and  }\lambda(H,G)\neq 0\}.$
If $G$ satisfies the $(\lambda,\mu)$-property, then for any $t\in \mathbb N,$ the vector $$\beta_t(G)=(\beta(H,t))_{[H] \in \mathcal C^*(G)}$$
is an integer solution of the linear equation
\begin{equation}\label{equabeta}
\sum_{\mathcal [H] \in C^*(G)}\lambda(H,G)x_H=0.
\end{equation}
One could investigate about the dimension of the vector space generate by the vectors $\beta_t(G),$ $t\in \mathbb N.$
For example, if $G=A_5$, then
we may order the elements of $\mathcal C^*(G)$ so that
$H_1=A_4,$ $H_2=S_3,$ $H_3=D_5,$ $H_4=\langle (1,2,3) \rangle,$
$H_5=\langle (1,2)(3,4)\rangle,$ $H_6=1.$
Then (\ref{equabeta}) can be written in the form
$$\sum_{[H] \in \mathcal C^*(G)}\lambda(H,G)x_H=-x_{H_1}-x_{H_2}-x_{H_3}+x_{H_4}+2x_{H_5}-x_{H_6}$$ and
$$\begin{aligned}
\beta_1(G)&=(24,24,20,39,44,59),\\
\beta_2(G)&=(84,54,50,99,74,59),\\
\beta_3(G)&=(264,114,110,279,134,59),\\
\beta_4(G)&=(804,234,230,819,254,59),\\
\beta_5(G)&=(2424,474,470,2439,494,59),\\
\beta_6(G)&=(7284,954,950,7299,974,59).
\end{aligned}$$
The first three vectors $\beta_1(G), \beta_2(G), \beta_3(G)$ are linearly independent, while $\beta_4(G),$ $\beta_5(G)$ and $\beta_6(G)$ can be obtained as linear combinations of $\beta_1(G), \beta_2(G), \beta_3(G).$

The situation is completely different when $G=S_3$. We may order the elements of $C^*(G)$ so that  $H_1=\langle (1,2,3)\rangle,$
$H_2=\langle (1,2) \rangle,$ $H_3=1.$  The equation (\ref{equabeta}) has in this case the form $x_{H_1}+x_{H_2}-x_{H_3}=0$ and
$\beta_t(G)=(0,2,2)$ independently on the choice of $t.$

Some properties of the vectors $\beta_t(G)$ are described in the following propositions.
\begin{prop}
	If $H \in \mathcal C^*(G),$ then $\beta(H,t)\geq 0$ with equality if and only if $G^\prime \leq H.$ In particular $\beta_t(G)$ is a non-negative vector and $\beta_t(G)=0$ if and only if $G$ is nilpotent.
\end{prop}
\begin{proof}
Notice that $\omega(H,t)\leq |G:N_G(H)|(|H|^t-1)+1.$
So
$$\begin{aligned}\beta(H,t) &\geq \frac{|H|^{t-1}|G||G^\prime  H|}{|G^\prime N_G(H)|}
-|G:N_G(H)|(|H|^t-1)-1\\&=
|H|^{t}|G:N_G(H)|\frac{|G^\prime \cap N_G(H)|}{|G^\prime\cap  H|}
-|G:N_G(H)|(|H|^t-1)-1\geq 0
\end{aligned}$$
with equality if and only if $H \geq G'.$
\end{proof}
\begin{prop}
	The vector $\beta_t(G)$ is independent on the choice of $t$ if and only if $G$ is a nilpotent group or a primitive Frobenius group, with cyclic Frobenius complement.
\end{prop}
\begin{proof}
	By the previous proposition, if $G$ is nilpotent then $\beta_t(G)$ is the zero vector for any $t\in \mathbb N,$ so we may assume that $G$ is not nilpotent. Assume that $\beta_t(G)$ is
	independent on the choice of $t.$
Let $H$ be a maximal non-normal subgroup of $G$. Then
$\alpha(H,t)=|H|^t\cdot u$ with $u=|G:H|$. Let
$H_1,\dots,H_u$ be the conjugates of $H$ in $G.$ For any $J\subseteq \{1,\dots,u\},$ let $\alpha_J=|\cap_{j\in J}H_j|.$ Then
$$\beta(H,t)=\sum_{J \neq  \{1,\dots,u\}}(-1)^{|J|+1}|\alpha_J|^t.$$
 We must have $\alpha_J=1$ for every choice of $J,$ otherwise $\lim_{t\to \infty}\beta(H,t)=\infty.$
Hence $H$ is a Frobenius complement and, since $H$ is a maximal subgroup, the Frobenius kernel $V$ is an irreducible $H$-module. 
Since $\beta(V,t)=|V|^t(|H^\prime|-1)$  does not depends on $t$, $H$ must be abelian, and consequently cyclic. So if $\beta_t(G)$ is independent of the choice of $t,$ then $G$ is a primitive Frobenius group with a cyclic Frobenius complement. Conversely assume $G=V\rtimes H,$ where $H$ is cyclic and $V$ and irreducible $H$-module.
If $X\in \mathcal C^*(G)$, then $\lambda(X,G)\neq 0,$ so $X$ is an intersection of maximal subgroups of $G$ and therefore either $V=G^\prime \leq X,$ or $X$ is conjugate to a subgroup of $H.$ In the first case $\beta(H,t)=0$. Assume
$X=K^v$ for some $K\leq H$ and $v \in V.$
Then $\beta(H,t)=|K|^t|V|-\omega(K,t)=
|K|^t|V|-(|V|(|K|^t-1)+1)=|V|-1.$
\end{proof}

\end{document}